\documentclass[a4paper,twoside,12pt]{article}
\usepackage[a4paper,nohead,includefoot,total={15cm,23.7cm},centering]{geometry}
\usepackage[all]{xy}
\usepackage{cite,proof,amsmath,amssymb,amsthm}

\newcommand{\shoenfield}{S}
\newcommand{\goedel}{D}
\newcommand{\negative}{N}
\newcommand{\bfi}{B}
\newcommand{\krivine}{K}
\newcommand{\sbfi}{U}

\newcommand{\ha}{\mathsf{HA}^\omega}
\newcommand{\pa}{\mathsf{PA}^\omega}
\newcommand{\haunlhd}{\mathsf{HA}^\omega_\unlhd}
\newcommand{\paunlhd}{\mathsf{PA}^\omega_\unlhd}
\newcommand{\blem}{\mathsf{B}\text{-}\mathsf{LEM}}
\newcommand{\bmac}{\mathsf{B}\text{-}\mathsf{mAC}}

\newcommand{\tforall}{\tilde{\forall}}
\newcommand{\texists}{\tilde{\exists}}
\newcommand{\defequiv}{\mathrel{\mathop:}\equiv}
\newcommand{\ul}[1]{\:\! \underline{#1} \:\!}

\newcommand{\citenospace}[1]{[\citen{#1}]}

\newtheorem{theorem}{Theorem}
\newtheorem{lemma}[theorem]{Lemma}

\theoremstyle{definition}
\newtheorem{definition}[theorem]{Definition}
\theoremstyle{remark}
\newtheorem{remark}[theorem]{Remark}
\newtheorem{notation}[theorem]{Notation}

\begin{document}

\title{Factorization of the Shoenfield-like bounded functional interpretation\footnote{Reprinted, in part, from the final ``Factorization of the Shoenfield-like bounded functional interpretation,'' in Notre Dame Journal of Formal Logic, volume 50, number 1, 2009, pages 53--60. Copyright 2008, University of Notre Dame. Used by permission of the publisher, Duke University Press.\newline
2001 Mathematics Subject Classification: 03F03, 03F10. Keywords: functional interpretation, negative translation, majorizability.}}
\author{Jaime Gaspar\footnote{Arbeitsgruppe Logik, Fachbereich Mathematik, Technische Universit\"{a}t Darmstadt, Schlossgartenstrasse 7, 64289 Darmstadt, Germany. \texttt{mail@jaimegaspar.com}, \texttt{www.jaimegaspar.com}.\newline
I am grateful for the suggestions of Ulrich Kohlenbach, Fernando Ferreira, and an anonymous referee. This work was financially supported by the Portuguese Funda\c{c}\~{a}o para a Ci\^{e}ncia e a Tecnologia, grant SFRH/BD/36358/2007.}}
\date{8 September 2010}
\maketitle

\begin{abstract}
  We adapt Streicher and Kohlenbach's proof of the factorization $S = KD$ of the Shoenfield translation $S$ in terms of Krivine's negative translation $K$ and the G\"{o}del functional interpretation $D$, obtaining a proof of the factorization $U = KB$ of Ferreira's Shoenfield-like bounded functional interpretation $U$ in terms of $K$ and Ferreira and Oliva's bounded functional interpretation $B$.
\end{abstract}

\section{Introduction}

In 1958, G\"{o}del \cite{godelDialectica}\nocite{godelCollectedWorksII} presented a functional interpretation $\goedel$ of Heyting arithmetic $\ha$ into itself (actually, into a quantifier-free theory, for foundational reasons). When composed with a negative translation $\negative$ of Peano arithmetic $\pa$ into $\ha$ (G\"{o}del \cite{godelNegativeTranslation}\nocite{godelCollectedWorksI}), it results in a two-step functional interpretation $\negative \goedel$ of $\pa$ into $\ha$ \cite{godelDialectica}. Nine years later, Shoenfield \cite{shoenfield} presented a one-step functional interpretation $\shoenfield$ of $\pa$ into $\ha$.

In 2007, Streicher and Kohlenbach \cite{streicherKohlenbach}, and independently Avigad \cite{avigad}, proved the factorization $\shoenfield = \krivine \goedel$ of $\shoenfield$ in terms of $\goedel$ and a negative translation $\krivine$ due to Streicher and Reus \cite{streicherReus}, inspired by Krivine \cite{krivine}.
\begin{equation*}
  \xymatrix@M=5pt{\pa \ar[r]^\krivine \ar@/_1pc/[rr]_\shoenfield & \ha \ar[r]^\goedel & \ha}
\end{equation*}

In 2005, Ferreira and Oliva \cite{ferreiraOliva} presented a functional interpretation $\bfi$ of Heyting arithmetic with majorizability $\haunlhd$ into itself. Like $\goedel$, when composed with a negative translation $\negative$ of Peano arithmetic with majorizability $\paunlhd$ into $\haunlhd$, it results in a two-step functional interpretation $\negative \bfi$ of $\paunlhd$ into $\haunlhd$ \cite{ferreiraOliva}. Two years later, Ferreira \cite{ferreira} presented a one-step functional interpretation $\sbfi$ of $\paunlhd$ into $\haunlhd$.

By adapting Streicher and Kohlenbach's proof, we obtain the factorization $\sbfi = \krivine \bfi$.
\begin{equation*}
  \xymatrix@M=5pt{\paunlhd \ar[r]^\krivine \ar@/_1pc/[rr]_\sbfi & \haunlhd \ar[r]^\bfi & \haunlhd}
\end{equation*}

\section{Framework}

\begin{definition}[\citenospace{ferreiraOliva,troelstra}]
  The \emph{Heyting arithmetic} $\ha$ that we consider is the usual Heyting arithmetic in all finite types, but with a minimal treatment of equality and no extensionality, following Anne Troelstra \cite{troelstra}.

  The \emph{Heyting arithmetic with majorizability} $\haunlhd$ is obtained from $\ha$ by
  \begin{enumerate}
    \item adding new atomic formulas $t \unlhd _\rho q$ for all finite types $\rho$ (where $t$ and $q$ are terms of type $\rho$);
    \item adding syntactically new \emph{bounded quantifications} $\forall x \unlhd_\rho t A$ and $\exists x \unlhd_\rho t A$ (where $A$ is a formula and the variable $x$ does not occur in the term $t$);
    \item adding the axioms
      \begin{equation*}
        \forall x \unlhd t A \leftrightarrow \forall x (x \unlhd t \to A), \qquad
        \exists x \unlhd t A \leftrightarrow \exists x (x \unlhd t \wedge A)
      \end{equation*}
      governing the bounded quantifications;
    \item adding the axioms and rule
      \begin{gather*}
        x \unlhd_0 y \leftrightarrow x \leq_0 y, \qquad
        x \unlhd y \to \forall u \unlhd v (xu \unlhd yv \wedge yu \unlhd yv), \\
        \vcenter{\infer{A_b \to t \unlhd q}{A_b \wedge u \unlhd v \to tu \unlhd qv \wedge qu \unlhd qv}}
      \end{gather*}
      governing the \emph{majorizability} symbol $\unlhd$ (where $\leq_0$ is the usual inequality between terms of type $0$, $A_b$ is a \emph{bounded formula}, that is, a formula with all quantifications bounded, and in the rule the variables $u$ and $v$ do not occur free in the formula $A_b$ neither in the terms $t$ and $q$);
    \item extending the induction axiom to the new formulas.
  \end{enumerate}
  This system is presented in detail in \cite{ferreiraOliva}.
\end{definition}

We will need the following notation.

\begin{notation}[\citenospace{ferreiraOliva}]
  An underlined letter $\ul{t}$ means a tuple (possibly empty) of terms $t_1,\ldots,t_n$. We use the abbreviations
  \begin{align*}
    \ul{t} \unlhd \ul{t} &\defequiv t_1 \unlhd t_1 \wedge \cdots \wedge t_n \unlhd t_n, \\
    \forall \ul{x} A &\defequiv \forall x_1 \cdots \forall x_n A, & \exists \ul{x} A &\defequiv \exists x_1 \cdots \exists x_n A, \\
    \forall \ul{x} \unlhd \ul{t} A &\defequiv \forall x_1 \unlhd t_1 \cdots \forall x_n \unlhd t_n A, & \quad \exists \ul{x} \unlhd \ul{t} A &\defequiv \exists x_1 \unlhd t_1 \cdots \exists x_n \unlhd t_n A, \\
    \tforall \ul{x} A &\defequiv \forall \ul{x} (\ul{x} \unlhd \ul{x} \to A), & \texists \ul{x} A &\defequiv \exists \ul{x} (\ul{x} \unlhd \ul{x} \wedge A), \\
    \tforall \ul{x} \unlhd \ul{t} A &\defequiv \forall \ul{x} \unlhd \ul{t} (\ul{x} \unlhd \ul{x} \to A), & \texists \ul{x} \unlhd \ul{t} A &\defequiv \exists \ul{x} \unlhd \ul{t} (\ul{x} \unlhd \ul{x} \wedge A).
  \end{align*}
\end{notation}

We consider two logical principles.

\begin{definition}
  The \emph{law of excluded middle for bounded formulas} $\blem$ is the principle
  \begin{equation*}
    A_b \vee \neg A_b,
  \end{equation*}
  where $A_b$ is a bounded formula.
\end{definition}

\begin{definition}[\citenospace{ferreira}]
  The \emph{monotone bounded choice} $\bmac$ is the principle
  \begin{equation*}
    \tforall \ul{x} \texists \ul{y} A_b(\ul{x},\ul{y}) \to \texists \ul{Y} \tforall \ul{x} \texists \ul{y} \unlhd \ul{Y}\ul{x} A_b(\ul{x},\ul{y}),
  \end{equation*}
  where $A_b$ is a bounded formula.
\end{definition}

\section{Negative translation and bounded functional interpretations}

For the convenience of the reader, we recall the definitions of $\krivine$, $\bfi$ and $\sbfi$.

\begin{definition}[\citenospace{avigad,krivine,streicherKohlenbach,streicherReus}]
  \emph{Krivine's negative translation} (extended to arithmetic with majorizability)\footnote{It still holds a soundness theorem $\paunlhd \vdash A \Rightarrow \haunlhd \vdash A^\krivine$ and a characterization theorem $\paunlhd \vdash A \leftrightarrow A^\krivine$.} $A^\krivine$ of a formula $A$ of $\paunlhd$ based on $\neg,\vee,\forall\unlhd,\forall$ is $A^\krivine \defequiv \neg A_\krivine$, where $A_\krivine$ is defined by induction on the complexity of formulas.
  \begin{enumerate}
    \item If $A$ is an atomic formula, then $A_\krivine \defequiv \neg A$.
    \item $(\neg A)_\krivine \defequiv \neg A_\krivine$.
    \item $(A \vee B)_\krivine \defequiv A_\krivine \wedge B_\krivine$.
    \item $(\forall x \unlhd t A)_\krivine \defequiv \exists x \unlhd t A_\krivine$.
    \item $(\forall x A)_\krivine \defequiv \exists x A_\krivine$.
  \end{enumerate}
  If we consider $\wedge$ a primitive symbol, then:
  \begin{enumerate}
    \setcounter{enumi}{5}
    \item $(A \wedge B)_\krivine \defequiv A_\krivine \vee B_\krivine$.
  \end{enumerate}
\end{definition}

\begin{definition}[\citenospace{ferreiraOliva}] The \emph{bounded functional interpretation} $A^\bfi$ of a formula $A$ of $\haunlhd$ based on $\bot,\wedge,\vee,\to,\forall\unlhd,\exists\unlhd,\forall,\exists$ is defined by induction on the complexity of formulas.
  \begin{enumerate}
    \item If $A$ is an atomic formula, then $A^\bfi \defequiv \texists \ul{x} \tforall \ul{y} A_\bfi(\ul{x},\ul{y}) \defequiv A$, where $\ul{x}$ and $\ul{y}$ are empty tuples.
  \end{enumerate}
  If $A^\bfi \equiv \texists \ul{x} \tforall \ul{y} A_\bfi(\ul{x},\ul{y})$ and $B^\bfi \equiv \texists \ul{x'} \tforall \ul{y'} B_\bfi(\ul{x'},\ul{y'})$, then:
  \begin{enumerate}
    \setcounter{enumi}{1}
    \item $(A \wedge B)^\bfi \defequiv \texists \ul{x},\ul{x'} \tforall \ul{y},\ul{y'} (A \wedge B)_\bfi(\ul{x},\ul{x'},\ul{y},\ul{y'}) \defequiv$\\$\texists \ul{x},\ul{x'} \tforall \ul{y},\ul{y'} [A_\bfi(\ul{x},\ul{y}) \wedge B_\bfi(\ul{x'},\ul{y'})]$;
    \item $(A \vee B)^\bfi \defequiv \texists \ul{x},\ul{x'} \tforall \ul{y},\ul{y'} (A \vee B)_\bfi(\ul{x},\ul{x'},\ul{y},\ul{y'}) \defequiv$\\$\texists \ul{x},\ul{x'} \tforall \ul{y},\ul{y'} [\tforall \ul{\tilde{y}} \unlhd \ul{y} A_\bfi(\ul{x},\ul{\tilde{y}}) \vee \tforall \ul{\tilde{y}'} \unlhd \ul{y'} B_\bfi(\ul{x'},\ul{\tilde{y}'})]$;
    \item $(A \to B)^\bfi \defequiv \texists \ul{X'},\ul{Y} \tforall \ul{x},\ul{y'} (A \to B)_\bfi(\ul{X'},\ul{Y},\ul{x},\ul{y'}) \defequiv$\\$\texists \ul{X'},\ul{Y} \tforall \ul{x},\ul{y'} [\tforall \ul{y} \unlhd \ul{Y} \ul{x} \ul{y'} A_\bfi(\ul{x},\ul{y}) \to B_\bfi(\ul{X'}\ul{x},\ul{y'})]$;
    \item $(\forall z \unlhd t A)^\bfi \defequiv \texists \ul{x} \tforall \ul{y} (\forall z \unlhd t A)_\bfi(\ul{x},\ul{y}) \defequiv \texists \ul{x} \tforall \ul{y} \forall z \unlhd t A_\bfi(\ul{x},\ul{y})$;
    \item $(\exists z \unlhd t A)^\bfi \defequiv \texists \ul{x} \tforall \ul{y} (\exists z \unlhd t A)_\bfi(\ul{x},\ul{y}) \defequiv \texists \ul{x} \tforall \ul{y} \exists z \unlhd t \tforall \ul{\tilde{y}} \unlhd \ul{y} A_\bfi(\ul{x},\ul{\tilde{y}})$;
    \item $(\forall z A)^\bfi \defequiv \texists \ul{X} \tforall w,\ul{y} (\forall z A)_\bfi(\ul{X},w,\ul{y}) \defequiv \texists \ul{X} \tforall w,\ul{y} \forall z \unlhd w A_\bfi(\ul{X}w,\ul{y})$;
    \item $(\exists z A)^\bfi \defequiv \texists w,\ul{x} \tforall \ul{y} (\exists z A)_\bfi(w,\ul{x},\ul{y}) \defequiv \texists w,\ul{x} \tforall \ul{y} \exists z \unlhd w \tforall \ul{\tilde{y}} \unlhd \ul{y} A_\bfi(\ul{x},\ul{\tilde{y}})$.
  \end{enumerate}
\end{definition}

\begin{remark}[\citenospace{ferreiraOliva}]
  From 1 and 4 we conclude that if $A^\bfi \equiv \texists \ul{x} \tforall \ul{y} A_\bfi(\ul{x},\ul{y})$, then $(\neg A)^\bfi \equiv \texists \ul{Y} \tforall \ul{x} (\neg A)_\bfi(\ul{Y},\ul{x}) \equiv \texists \ul{Y} \tforall \ul{x} \neg \tforall \ul{y} \unlhd \ul{Y}\ul{x} A_\bfi(\ul{x},\ul{y})$.
\end{remark}

\begin{remark}[\citenospace{ferreiraOliva}]
  We can prove by induction on the complexity of formulas that $A_\bfi(\ul{x},\ul{y})$ is a bounded formula.
\end{remark}

\begin{definition}[\citenospace{ferreira}]
  The \emph{Shoenfield-like bounded functional interpretation} $A^\sbfi$ of a formula $A$ of $\paunlhd$ based on $\neg,\vee,\forall\unlhd,\forall$ is defined by induction on the complexity of formulas.
  \begin{enumerate}
    \item If $A$ is an atomic formula, then $A^\sbfi \defequiv \tforall \ul{x} \texists \ul{y} A_\sbfi(\ul{x},\ul{y}) \defequiv A$, where $\ul{x}$ and $\ul{y}$ are empty tuples.
  \end{enumerate}
  If $A^\sbfi \equiv \tforall \ul{x} \texists \ul{y} A_\sbfi(\ul{x},\ul{y})$ e $B^\sbfi \equiv \tforall \ul{x'} \texists \ul{y'} B_\sbfi(\ul{x'},\ul{y'})$, then:
  \begin{enumerate}
    \setcounter{enumi}{1}
    \item $(\neg A)^\sbfi \defequiv \tforall \ul{Y} \texists \ul{x} (\neg A)_\sbfi(\ul{Y},\ul{x}) \defequiv \tforall \ul{Y} \texists \ul{x} \texists \ul{\tilde{x}} \unlhd \ul{x} \neg A_\sbfi(\ul{\tilde{x}},\ul{Y}\ul{\tilde{x}})$;
    \item $(A \vee B)^\sbfi \defequiv \tforall \ul{x},\ul{x'} \texists \ul{y},\ul{y'} (A \vee B)_\sbfi(\ul{x},\ul{x'},\ul{y},\ul{y'}) \defequiv$\\$\tforall \ul{x},\ul{x'} \texists \ul{y},\ul{y'} [A_\sbfi(\ul{x},\ul{y}) \vee B_\sbfi(\ul{x'},\ul{y'})]$;
    \item $(\forall z \unlhd t A)^\sbfi \defequiv \tforall \ul{x} \texists \ul{y} (\forall z \unlhd t A)_\sbfi(\ul{x},\ul{y}) \defequiv \tforall \ul{x} \texists \ul{y} \forall z \unlhd t A_\sbfi(\ul{x},\ul{y})$;
    \item $(\forall z A)^\sbfi \defequiv \tforall w,\ul{x} \texists \ul{y} (\forall z A)_\sbfi(w,\ul{x},\ul{y}) \defequiv \tforall w,\ul{x} \texists \ul{y} \forall z \unlhd w A_\sbfi(\ul{x},\ul{y})$.
  \end{enumerate}
  If we consider $\wedge$ a primitive symbol, then:
  \begin{enumerate}
    \setcounter{enumi}{5}
    \item $(A \wedge B)^\sbfi \defequiv \tforall \ul{x},\ul{x'} \texists \ul{y},\ul{y'} (A \wedge B)_\sbfi(\ul{x},\ul{x'},\ul{y},\ul{y'}) \defequiv$\\$\tforall \ul{x},\ul{x'} \texists \ul{y},\ul{y'} [A_\sbfi(\ul{x},\ul{y}) \wedge B_\sbfi(\ul{x'},\ul{y'})]$.
  \end{enumerate}
\end{definition}

\begin{remark}[\citenospace{ferreira}]
  We can also prove by induction on the complexity of formulas that $A_\sbfi(\ul{x},\ul{y})$ is a bounded formula.
\end{remark}

$\sbfi$ is monotone on the second tuple of the variables, in the following sense.

\begin{lemma}[monotonicity of $\sbfi$\cite{ferreira}]
  $\haunlhd \vdash \forall \ul{x} \forall \ul{y} \forall \ul{\tilde{y}} \unlhd \ul{y} [A_\sbfi(\ul{x},\ul{\tilde{y}}) \to A_\sbfi(\ul{x},\ul{y})]$.
\end{lemma}

\section{Factorization}

We want to prove $A^\sbfi \leftrightarrow (A^\krivine)^\bfi$ by induction on the complexity of formulas. Because it isn't $A^\krivine$ but $A_\krivine$ that is defined by induction on the complexity of formulas, it would be better to write $A^\sbfi \leftrightarrow (\neg A_\krivine)^\bfi$. If $A^\sbfi \equiv \tforall \ul{x} \texists \ul{y} A_\sbfi(\ul{x},\ul{y})$ and $(A_\krivine)^\bfi \equiv \texists \ul{x'} \tforall \ul{y'} (A_\krivine)_\bfi(\ul{x'},\ul{y'})$, then using $\bmac$ in the first equivalence and the monotonicity of $\sbfi$ in the second equivalence, we have
\begin{align}
  A^\sbfi &\equiv \tforall \ul{x} \texists \ul{y} A_\sbfi(\ul{x},\ul{y}) \nonumber \\
  &\leftrightarrow \texists \ul{Y} \tforall \ul{x} \texists \ul{y} \unlhd \ul{Y} \ul{x} A_\sbfi(\ul{x},\ul{y}) \nonumber \\
  &\leftrightarrow \texists \ul{Y} \tforall \ul{x} A_\sbfi(\ul{x},\ul{Y} \ul{x}), \label{equationMotivation1} \\
  (\neg A_\krivine)^\bfi &\equiv \texists \ul{Y'} \tforall \ul{x'} \neg \tforall \ul{y'} \unlhd \ul{Y'} \ul{x'} (A_\krivine)_\bfi(\ul{x'},\ul{y'}). \label{equationMotivation2}
\end{align}
The comparison of formulas (\ref{equationMotivation1}) and (\ref{equationMotivation2}) suggests that we first prove $A_\sbfi(\ul{x},\ul{Y}\ul{x}) \leftrightarrow \neg \tforall \ul{y} \unlhd \ul{Y} \ul{x} (A_\krivine)_\bfi(\ul{x},\ul{y})$, or even better, $A_\sbfi(\ul{x},\ul{y}) \leftrightarrow \neg \tforall \ul{\tilde{y}} \unlhd \ul{y} (A_\krivine)_\bfi(\ul{x},\ul{\tilde{y}})$. Then, by the above argument, we would have $A^\sbfi \leftrightarrow (A^\krivine)^\bfi$.

The factorization proof is almost the straightforward adaptation of Streicher and Kohlenbach's proof but with two tweaks.
\begin{enumerate}
  \item Instead of proving $A_\sbfi(\ul{x},\ul{y}) \leftrightarrow \neg (A_\krivine)_\bfi(\ul{x},\ul{y})$, along the lines of Streicher and Kohlenbach's proof, we prove $A_\sbfi(\ul{x},\ul{y}) \leftrightarrow \neg \tforall \ul{\tilde{y}} \unlhd \ul{y} (A_\krivine)_\bfi(\ul{x},\ul{\tilde{y}})$, where the appearance of the quantification $\tforall \ul{\tilde{y}} \unlhd \ul{y}$ is explained by the above argument.
  \item In proving $A_\sbfi(\ul{x},\ul{y}) \leftrightarrow \neg \tforall \ul{\tilde{y}} \unlhd \ul{y} (A_\krivine)_\bfi(\ul{x},\ul{\tilde{y}})$ we need the hypothesis $\ul{x} \unlhd \ul{x} \wedge \ul{y} \unlhd \ul{y}$ for technical reasons explained in footnotes.
\end{enumerate}
\begin{theorem}[factorization $\sbfi = \krivine \bfi$]
  We have
  \begin{align}
    \haunlhd + \blem &\vdash \tforall \ul{Y},\ul{x} [A_\sbfi(\ul{x},\ul{Y}\ul{x}) \leftrightarrow (A^\krivine)_\bfi(\ul{Y},\ul{x})], \label{equationFactorization1} \\
    \haunlhd + \blem + \bmac &\vdash A^\sbfi \leftrightarrow (A^\krivine)^\bfi. \label{equationFactorization2}
  \end{align}
\end{theorem}

\begin{proof}
  Step 1.~First we prove
  \begin{equation}
    \haunlhd + \blem \vdash \tforall \ul{x},\ul{y} [A_\sbfi(\ul{x},\ul{y}) \leftrightarrow \neg \tforall \ul{\tilde{y}} \unlhd \ul{y} (A_\krivine)_\bfi(\ul{x},\ul{\tilde{y}})] \label{equationFactorization3}
  \end{equation}
  by induction on the complexity of formulas.

  \noindent Let us consider the case of atomic formulas $A$. Using $\blem$ in the equivalence, we have
  \begin{align*}
    A_\sbfi &\equiv A \\
    &\leftrightarrow \neg\neg A \\
    &\equiv \neg (A_\krivine)_\bfi.
  \end{align*}

  \noindent Let us now consider the case of negation $\neg A$. Assume $\ul{Y} \unlhd \ul{Y}$ and $\ul{x} \unlhd \ul{x}$. Using the induction hypothesis in the first equivalence and $\blem$ in the second equivalence, we have
  \begin{align*}
    (\neg A)_\sbfi(\ul{Y},\ul{x}) &\equiv \texists \ul{\tilde{x}} \unlhd \ul{x} \neg A_\sbfi(\ul{\tilde{x}},\ul{Y}\ul{\tilde{x}}) \\
    &\leftrightarrow \texists \ul{\tilde{x}} \unlhd \ul{x} \neg\neg \tforall \ul{y} \unlhd \ul{Y} \ul{\tilde{x}} (A_\krivine)_\bfi(\ul{\tilde{x}},\ul{y}) \\
    &\leftrightarrow \neg \tforall \ul{\tilde{x}} \unlhd \ul{x} \neg \tforall \ul{y} \unlhd \ul{Y} \ul{\tilde{x}} (A_\krivine)_\bfi(\ul{\tilde{x}},\ul{y}) \\
    &\equiv \neg \tforall \ul{\tilde{x}} \unlhd \ul{x} [(\neg A)_\krivine]_\bfi(\ul{Y},\ul{\tilde{x}}).
  \end{align*}

  \noindent Let us now consider the case of disjunction $A \vee B$. Assume $\ul{x} \unlhd \ul{x}$, $\ul{x'} \unlhd \ul{x'}$, $\ul{y} \unlhd \ul{y}$, and $\ul{y'} \unlhd \ul{y'}$. Using the induction hypothesis in the first equivalence, $\blem$ in the second equivalence, and intuitionistic logic in the third equivalence,\footnote{The rule for conversion to prenex normal form $\forall u \unlhd v (C \wedge D) \to \forall u \unlhd v C \wedge D$ (where the variable $u$ does not occur free in the formula $D$), despite its innocuous look, does not hold without the hypothesis $v \unlhd v$. So we need to use the hypothesis $\ul{x} \unlhd \ul{x} \wedge \ul{y} \unlhd \ul{y}$ in the proof.} we have
  \begin{align*}
    (A \vee B)_\sbfi(\ul{x},\ul{x'},\ul{y},\ul{y'}) &\equiv A_\sbfi(\ul{x},\ul{y}) \vee B_\sbfi(\ul{x'},\ul{y'}) \\
    &\leftrightarrow \neg \tforall \ul{\tilde{y}} \unlhd \ul{y} (A_\krivine)_\bfi(\ul{x},\ul{\tilde{y}}) \vee \neg \tforall \ul{\tilde{y}'} \unlhd \ul{y'} (B_\krivine)_\bfi(\ul{x'},\ul{\tilde{y}'}) \\
    &\leftrightarrow \neg[\tforall \ul{\tilde{y}} \unlhd \ul{y} (A_\krivine)_\bfi(\ul{x},\ul{\tilde{y}}) \wedge \tforall \ul{\tilde{y}'} \unlhd \ul{y'} (B_\krivine)_\bfi(\ul{x'},\ul{\tilde{y}'})] \\
    &\leftrightarrow \neg \tforall \ul{\tilde{y}},\ul{\tilde{y}'} \unlhd \ul{y},\ul{y'} [(A_\krivine)_\bfi(\ul{x},\ul{\tilde{y}}) \wedge (B_\krivine)_\bfi(\ul{x'},\ul{\tilde{y}'})] \\
    &\equiv \neg \tforall \ul{\tilde{y}},\ul{\tilde{y}'} \unlhd \ul{y},\ul{y'} [(A \vee B)_\krivine]_\bfi(\ul{x},\ul{x'},\ul{\tilde{y}},\ul{\tilde{y}'}).
  \end{align*}

  \noindent Let us now consider the case of bounded universal quantification $\forall z \unlhd t A$. Assume $\ul{x} \unlhd \ul{x}$ and $\ul{y} \unlhd \ul{y}$. Using the induction hypothesis in the first equivalence and intuitionistic logic in the second and third\footnote{Probably the easiest way to prove the third equivalence is to prove
  \begin{equation*}
    \exists z \unlhd t \tforall \ul{\tilde{y}} \unlhd \ul{y} (A_\krivine)_\bfi(\ul{x},\ul{\tilde{y}}) \leftrightarrow \tforall \ul{\hat{y}} \unlhd \ul{y} \exists z \unlhd t \tforall \ul{\tilde{y}} \unlhd \ul{\hat{y}} (A_\krivine)_\bfi(\ul{x},\ul{\tilde{y}}).
  \end{equation*}
  To prove the right-to-left implication, we just take $\ul{\hat{y}} = \ul{y}$, which we can do because $\ul{y} \unlhd \ul{y}$. So here again we need to use the hypothesis $\ul{x} \unlhd \ul{x} \wedge \ul{y} \unlhd \ul{y}$.} equivalences, we have
  \begin{align*}
    (\forall z \unlhd t A)_\sbfi(\ul{x},\ul{y}) &\equiv \forall z \unlhd t A_\sbfi(\ul{x},\ul{y}) \\
    &\leftrightarrow \forall z \unlhd t \neg \tforall \ul{\tilde{y}} \unlhd \ul{y} (A_\krivine)_\bfi(\ul{x},\ul{\tilde{y}}) \\
    &\leftrightarrow \neg \exists z \unlhd t \tforall \ul{\tilde{y}} \unlhd \ul{y} (A_\krivine)_\bfi(\ul{x},\ul{\tilde{y}}) \\
    &\leftrightarrow \neg \tforall \ul{\hat{y}} \unlhd \ul{y} \exists z \unlhd t \tforall \ul{\tilde{y}} \unlhd \ul{\hat{y}} (A_\krivine)_\bfi(\ul{x},\ul{\tilde{y}}) \\
    &\equiv \neg \tforall \ul{\hat{y}} \unlhd \ul{y} [(\forall z \unlhd t A)_\krivine]_\bfi(\ul{x},\ul{\hat{y}}).
  \end{align*}

  \noindent Finally, let us consider the case of unbounded universal quantification $\forall z A$. Assume $w \unlhd w$, $\ul{x} \unlhd \ul{x}$, and $\ul{y} \unlhd \ul{y}$. Using the induction hypothesis in the first equivalence and intuitionistic logic in the second and third equivalences, we have
  \begin{align*}
    (\forall z A)_\sbfi(w,\ul{x},\ul{y}) &\equiv \forall z \unlhd w A_\sbfi(\ul{x},\ul{y}) \\
    &\leftrightarrow \forall z \unlhd w \neg \tforall \ul{\tilde{y}} \unlhd \ul{y} (A_\krivine)_\bfi(\ul{x},\ul{\tilde{y}}) \\
    &\leftrightarrow \neg \exists z \unlhd w \tforall \ul{\tilde{y}} \unlhd \ul{y} (A_\krivine)_\bfi(\ul{x},\ul{\tilde{y}}) \\
    &\leftrightarrow \neg \tforall \ul{\hat{y}} \unlhd \ul{y} \exists z \unlhd w \tforall \ul{\tilde{y}} \unlhd \ul{\hat{y}} (A_\krivine)_\bfi(\ul{x},\ul{\tilde{y}}) \\
    &\equiv \neg \tforall \ul{\hat{y}} \unlhd \ul{y} [(\forall z A)_\krivine]_\bfi(w,\ul{x},\ul{\hat{y}}).
  \end{align*}

  \noindent In case we consider $\wedge$ a primitive symbol, let us now see the case of conjunction $A \wedge B$. Assume $\ul{x} \unlhd \ul{x}$, $\ul{x'} \unlhd \ul{x'}$, $\ul{y} \unlhd \ul{y}$, and $\ul{y'} \unlhd \ul{y'}$. Using the induction hypothesis in the first equivalence and intuitionistic logic in the second and third equivalences, we have
  \begin{align*}
    (A \wedge B)_\sbfi(\ul{x},\ul{x'},\ul{y},\ul{y'}) &\equiv A_\sbfi(\ul{x},\ul{y}) \wedge B_\sbfi(\ul{x'},\ul{y'}) \\
    &\leftrightarrow \neg \tforall \ul{\tilde{y}} \unlhd \ul{y} (A_\krivine)_\bfi(\ul{x},\ul{\tilde{y}}) \wedge \neg \tforall \ul{\tilde{y}'} \unlhd \ul{y'} (B_\krivine)_\bfi(\ul{x'},\ul{\tilde{y}'}) \\
    &\leftrightarrow \neg [\tforall \ul{\tilde{y}} \unlhd \ul{y} (A_\krivine)_\bfi(\ul{x},\ul{\tilde{y}}) \vee \tforall \ul{\tilde{y}'} \unlhd \ul{y'} (B_\krivine)_\bfi(\ul{x'},\ul{\tilde{y}'})] \\
    &\leftrightarrow \neg \tforall \ul{\hat{y}},\ul{\hat{y}'} \unlhd \ul{y},\ul{y'} [\tforall \ul{\tilde{y}} \unlhd \ul{\hat{y}} (A_\krivine)_\bfi(\ul{x},\ul{\tilde{y}}) \vee {} \\
    & \phantom{{} \leftrightarrow \neg {}} \tforall \ul{\tilde{y}'} \unlhd \ul{\hat{y}'} (B_\krivine)_\bfi(\ul{x'},\ul{\tilde{y}'})] \\
    &\equiv \neg \tforall \ul{\hat{y}},\ul{\hat{y}'} \unlhd \ul{y},\ul{y'} [(A \wedge B)_\krivine]_\bfi(\ul{x},\ul{x'},\ul{\hat{y}},\ul{\hat{y}'}).
  \end{align*}

  Step 2.~Now we prove (\ref{equationFactorization1}). Assume $\ul{Y} \unlhd \ul{Y}$ and $\ul{x} \unlhd \ul{x}$. Using (\ref{equationFactorization3}) in the equivalence, we have
  \begin{align*}
    A_\sbfi(\ul{x},\ul{Y}\ul{x}) &\leftrightarrow \neg \tforall \ul{y} \unlhd \ul{Y}\ul{x} (A_\krivine)_\bfi(\ul{x},\ul{y}) \\
    &\equiv (\neg A_\krivine)_\bfi(\ul{Y},\ul{x}) \\
    &\equiv (A^\krivine)_\bfi(\ul{Y},\ul{x}).
  \end{align*}

  Step 3.~Finally, we prove (\ref{equationFactorization2}). Using $\bmac$ in the first equivalence, the monotonicity of $\sbfi$ in the second equivalence and (\ref{equationFactorization1}) in the third equivalence, we have
  \begin{align*}
    A^\sbfi &\equiv \tforall \ul{x} \texists \ul{y} A_\sbfi(\ul{x},\ul{y}) \\
    &\leftrightarrow \texists \ul{Y} \tforall \ul{x} \texists \ul{y} \unlhd \ul{Y}\ul{x} A_\sbfi(\ul{x},\ul{y}) \\
    &\leftrightarrow \texists \ul{Y} \tforall \ul{x} A_\sbfi(\ul{x},\ul{Y}\ul{x}) \\
    &\leftrightarrow \texists \ul{Y} \tforall \ul{x} (A^\krivine)_\bfi(\ul{Y},\ul{x}) \\
    &\equiv (A^\krivine)^\bfi. \qedhere
  \end{align*}
\end{proof}

\bibliography{referenc}
\bibliographystyle{plain}

\end{document}